\documentclass[pdflatex,sn-mathphys-num]{sn-jnl}
\usepackage{mathrsfs}

\usepackage{graphicx}%
\usepackage{multirow}%
\usepackage{amsmath,amssymb,amsfonts}%
\usepackage{amsthm}%
\usepackage[title]{appendix}%
\usepackage{xcolor}%
\usepackage{letltxmacro}
\usepackage{textcomp}%
\usepackage{manyfoot}%
\usepackage{booktabs}%

\usepackage{algorithm}
\usepackage{algorithmicx}
\usepackage{algpseudocode} 
\algrenewcommand\algorithmicrequire{\textbf{Input:}}
\algrenewcommand\algorithmicensure{\textbf{Output:}}
\numberwithin{algorithm}{section}

\usepackage{listings}%

\theoremstyle{thmstyleone}%
\newtheorem{theorem}{Theorem}[section]
\newtheorem{proposition}[theorem]{Proposition}%
\newtheorem{assumption}[theorem]{Assumption}%
\newtheorem{lemma}[theorem]{Lemma}

\theoremstyle{thmstyletwo}%
\newtheorem{remark}[theorem]{Remark}%

\theoremstyle{thmstylethree}%

\begin{document}

\title[Article Title]{Entropic Optimal Transport Problem with Convex Functional Cost}


\author[1]{\fnm{Anna} \sur{Kazeykina}}\email{anna.kazeykina@math.u-psud.fr}
\author[2]{\fnm{Zhenjie} \sur{Ren}}\email{zhenjie.ren@univ-evry.fr}
\author[3]{\fnm{Xiaozhen} \sur{Wang}}\email{xiaozhen.wang@dauphine.psl.eu}
\author[4]{\fnm{Yufei} \sur{Zhang}}\email{yufei.zhang@imperial.ac.uk}

\affil[1]{\orgdiv{LMO}, \orgname{Université Paris-Saclay}, \orgaddress{\country{France}}}
\affil[2]{\orgdiv{LaMME}, \orgname{Université Evry Paris-Saclay}, \orgaddress{\country{France}}}
\affil[3]{\orgdiv{CEREMADE}, \orgname{Université Paris Dauphine-PSL}, \orgaddress{\country{France}}}
\affil[4]{\orgdiv{Department of Mathematics}, \orgname{Imperial College London}, \orgaddress{\country{United Kingdom}}}










\abstract{We study an entropic optimal transport problem in which the transport plan is penalized by a nonlinear convex functional acting  on the coupling. We establish existence, uniqueness, and uniform a priori bounds for minimizers, and we show that each minimizer satisfies a fixed-point first-order optimality system associated with an exponentially tilted reference measure.
Building on this variational structure, we introduce the Sinkhorn–Frank–Wolfe (SFW) flow, prove its global well-posedness, and derive an energy–dissipation inequality yielding exponential convergence toward the unique optimal transport plan. As an application, we implement the SFW algorithm to solve an optimal routing problem for unmanned aerial vehicles with congestion aversion. }




\maketitle

\section{Introduction}\label{sec1}


Optimal transport (OT) provides a variational method to compare probability measures by minimizing a transportation cost subject to marginal constraints \cite{Villani2009,santambrogio2015optimal,peyre2019}. Since the Kantorovich formulation, OT has become a central tool in analysis, probability, and optimization, with applications ranging from geometric data analysis and imaging \cite{peyre2019,feydy2017optimal,sarlin2020superglue} to financial calibration via martingale couplings \cite{henry2019martingale,chen2024EMOT} and traffic modeling \cite{Jimenez2008,Ekeland2007}. The dynamic formulation of Benamou and Brenier interprets transport as a time-dependent flow satisfying a continuity equation and links OT to PDE and control \cite{benamou-brenier}.

Entropic regularization introduces a relative-entropy penalty into the OT objective, leading to the entropic optimal transport (EOT) problem and connecting OT with Schr\"odinger bridge formulations \cite{leonard2014survey}. This modification yields a strictly convex objective and enables efficient numerical solutions via Sinkhorn’s scaling algorithm \cite{cuturi2013sinkhorn}, stabilized variants for small regularization \cite{schmitzer2019stabilized}, and extensions such as Greenkhorn and multi-marginal scaling \cite{altschuler2017near,chizat2018scaling}. From an analytic standpoint, recent works have clarified the convergence of entropic potentials to Kantorovich potentials \cite{nutzWiesel2021potentials}, established quantitative stability of entropic couplings \cite{ghosalNutz2022stability}, and derived sharp differentiation formulas for the entropic cost \cite{conforti2021timeDerivative}.

We work on two Polish spaces $(\mathcal X,\mathcal B(\mathcal X))$ and 
$(\mathcal Y,\mathcal B(\mathcal Y))$, endowed with their Borel $\sigma$–algebras. Given probability measures $\mu\in\mathcal P(\mathcal X)$ and $\nu\in\mathcal P(\mathcal Y)$, the set of admissible couplings with prescribed marginals is
\begin{equation}
\label{eq:marginal-coupling}
\begin{aligned}
\Pi(\mu,\nu) := \Bigl\{\pi\in\mathcal P(\mathcal X\times\mathcal Y):
\pi(A\times\mathcal Y)=\mu(A), &  \pi(\mathcal X\times B)=\nu(B),\\
&\forall A\in\mathcal B(\mathcal X),\,B\in\mathcal B(\mathcal Y)\Bigr\}.
\end{aligned}
\end{equation}

We adopt the reference measure viewpoint for entropic transport. Fix a probability measure $R\in\mathcal P(\mathcal X\times\mathcal Y)$ with $R\ll\mu\otimes\nu$, and consider the static Schr\"odinger problem
\begin{equation}\label{eq:SB}
    \inf_{\pi\in\Pi(\mu,\nu)} H(\pi\,\|\,R).
\end{equation}
where $H(\cdot\,\|\,\cdot)$ denotes the relative entropy. This formulation encompasses the classical entropy–regularized optimal transport model: whenever the reference measure $R$ takes the exponentially tilted product form $dR \propto e^{-c/\varepsilon}\,d(\mu\otimes\nu)$ for some measurable cost $c$, problem~\eqref{eq:SB} is equivalent to the standard entropic OT with cost $c$.

In the classical formulation, the objective in \eqref{eq:SB} depends linearly on the coupling $\pi$, and the entropy acts solely as a convex regularization with respect to the fixed reference measure $R$. In many modelling frameworks, however, the effective cost of transporting mass may itself depend on the distribution $\pi$, for instance through interaction, congestion, capacity, or sparsity effects, so that the resulting cost is no longer purely linear in $\pi$ but exhibits a convex structure (see \cite{Paty2020Regularized, Mishra2021Manifold, Liu2023SparsityConstrained, Courty2017Optimal, Asadulaev2024Neural, Gallardo2024Congestion}). To accommodate such effects, we augment the entropic objective with an additional convex functional $F:\mathcal P(\mathcal X\times\mathcal Y)\to\mathbb R$ acting on the coupling measure, and study the resulting entropic optimal transport problem with convex functional cost.

The resulting variational formulation reads
\begin{equation}\label{eq:EOT-congestion}
    \inf_{\pi\in\Pi(\mu,\nu)} V(\pi) = \inf_{\pi\in\Pi(\mu,\nu)} 
    \bigl\{\, H(\pi\,\|\,R)\;+\;F(\pi)\,\bigr\}.
\end{equation}

This extension retains the entropic structure while allowing the effective cost to vary with $\pi$, thereby covering a broad class of regularized transport models arising in mean-field, interaction, and statistical-mechanics formulations (\cite{Mikami2009Optimal, Baudelet2024Deep}).

In this paper, our main contribution is to extend entropic optimal transport by incorporating a general convex functional $F$ acting on the coupling, leading to an entropic optimal transport problem with convex functional  cost, and to develop an algorithmic framework for solving it. The central idea is a Sinkhorn-Frank-Wolfe (SFW) flow: at each iteration, an \emph{inner loop} performs a single Sinkhorn scaling to solve the entropy-plus-linearized subproblem, while an \emph{outer loop} carries out a Frank--Wolfe step in the space of couplings (\cite{Jaggi2013Revisiting}). This construction preserves the computational scalability of Sinkhorn iterations, while the convexity of $F$ is handled through the outer linearization.

From an analytic perspective, the entropy–regularized functional admits a unique minimizer, and the SFW flow satisfies a quantitative energy dissipation inequality implying exponential convergence. The SFW flow inherits this contraction structure, yielding convergence of the iterates under natural compactness assumptions.

Our paper is structured as follows. In Section~\ref{sec:main_result}, we introduce the entropic optimal transport problem with convex functional cost, derive the first-order optimality system, and establish our main analytic results, including existence and uniqueness of minimizers, uniform a priori bounds, as well as global well-posedness and exponential convergence of the Sinkhorn-Frank-Wolfe flow. Section~\ref{sec:UAV} presents a concrete instance of the framework in the setting of unmanned aerial vehicle (UAV) relocation, where we specify the reference kernel, transport cost, and a congestion-type convex functional, implement the corresponding SFW algorithm, and report numerical experiments comparing the method with classical entropic transport baselines. All proofs of the main results, together with auxiliary estimates and technical lemmas, are collected in Section~\ref{sec:main_result_proofs}.

\subsection{Notation}

\textbf{State spaces: } Let $(\mathcal X,d_{\mathcal X})$ and $(\mathcal Y,d_{\mathcal Y})$ be Polish spaces. We write $\mathcal P(\mathcal Z)$ for the set of Borel probability measures on a Polish space $\mathcal Z$, and $\mu\otimes\nu$ for the product measure of $\mu\in\mathcal P(\mathcal X)$ and $\nu\in\mathcal P(\mathcal Y)$. The support of a measure $\rho$ is denoted by $\mathrm{supp}(\rho)$.

\noindent \textbf{Moment spaces and Wasserstein distance: } Given $p\in[1,\infty)$, we denote by $\mathcal P_p(\mathcal Z)$ the subset of measures
$\mu\in\mathcal P(\mathcal Z)$ with finite $p$--th moment; i.e.,
\[
\int d_{\mathcal Z}(z,\hat z)^p\,\mu(dz)<\infty
\quad\text{for some (and hence all) }\hat z\in\mathcal Z.
\]
We denote by $\mathcal P_\infty(\mathcal Z)$ the set of probability measures with bounded support. For $\rho,\eta\in\mathcal P_p(\mathcal Z)$ and $p\in[1,\infty)$, the $p$--Wasserstein distance is
\[
W_p(\rho,\eta)^p
:=\inf_{\pi\in\Pi(\rho,\eta)}
\int_{\mathcal Z\times\mathcal Z}
d_{\mathcal Z}(z,z')^p\,\pi(dz,dz').
\]
\textbf{Relative entropy: } For $\pi,\rho\in\mathcal P(\mathcal Z)$, the relative entropy of $\pi$ with respect to $\rho$ is
\[
H(\pi\,\|\,\rho)
=
\begin{cases}
\displaystyle
\int_{\mathcal Z}
\log\!\bigl(\tfrac{d\pi}{d\rho}\bigr)\,d\pi,
&\pi\ll\rho,\\[1ex]
+\infty,&\text{otherwise}.
\end{cases}
\]
It is always nonnegative and vanishes if and only if $\pi=\rho$.

\section{Main Results}
\label{sec:main_result}
In this section we present the main theoretical results. The analysis concerns the generalized entropic transport problem \ref{eq:EOT-congestion}:
\[
\inf_{\pi\in\Pi(\mu,\nu)}  V(\pi) = \inf_{\pi\in\Pi(\mu,\nu)} \bigl\{\,H(\pi\,\|\,R)\;+\;F(\pi)\,\bigr\},
\]
introduced in the previous section.  We begin by formulating the basic assumptions on the state spaces. These conditions ensure the weak compactness of $\Pi(\mu,\nu)$ and the finiteness of entropy and Wasserstein distances, hence providing a well-defined functional framework.
\begin{assumption}
\label{ass:bounded_domain}
The marginal spaces $\mathcal{X}$ and $\mathcal{Y}$ are bounded Polish spaces, 
that is, their diameters are finite:
\[
\mathrm{diam}(\mathcal{X}) := \sup_{x,x'\in\mathcal{X}} d_{\mathcal{X}}(x,x') < \infty,
\qquad
\mathrm{diam}(\mathcal{Y}) := \sup_{y,y'\in\mathcal{Y}} d_{\mathcal{Y}}(y,y') < \infty.
\]
Consequently, the product space $\mathcal{X}\times\mathcal{Y}$ 
equipped with the product metric also has finite diameter.
\end{assumption}
\begin{remark}
Under Assumption~\ref{ass:bounded_domain}, the supports of all admissible probability measures are compact. In particular, every marginal measure $\mu$ (resp.~$\nu$) supported on $\mathcal{X}$ (resp.~$\mathcal{Y}$) belongs to $\mathcal{P}_\infty(\mathcal{X})$ (resp.~$\mathcal{P}_\infty(\mathcal{Y})$), since all moments of order $q\ge1$ are finite on a bounded metric space. 
Consequently, for any coupling $\pi\in\Pi(\mu,\nu)$, with $ \pi \ll R$, the entropic term $H(\pi\|R)$ is well defined and finite.
\end{remark}

Having established the geometric conditions on $(\mathcal X,\mathcal Y)$ under Assumption~\ref{ass:bounded_domain}, we now turn to the analytic structure of the congestion functional $F$. The next assumption collects the convexity, differentiability, and quantitative regularity properties that will be required for the variational analysis of the entropic problem.

\begin{assumption}
\label{Ass:F_full}
Let $F:\mathcal{P}(\mathcal{X}\times\mathcal{Y})\to\mathbb{R}$ be a functional satisfying the following properties:
\begin{enumerate}[(i)]
\item Convexity and lower boundedness: $F$ is convex and non–negative, i.e.\ $F(\pi)\ge0$ for all $\pi\in\mathcal{P}(\mathcal{X}\times\mathcal{Y})$.
\item Differentiability and bounded oscillation: There exists a continuous map, called the \emph{functional linear derivative}, $\frac{\delta F}{\delta \pi}:
\mathcal{P}(\mathcal{X}\times\mathcal{Y})\times(\mathcal{X}\times\mathcal{Y})
\longrightarrow\;\mathbb{R}$,
such that for all $\pi,\pi'\in\mathcal{P}(\mathcal{X}\times\mathcal{Y})$,
\begin{equation*}
F(\pi)-F(\pi')
= \int_0^1 \int_{\mathcal{X}\times\mathcal{Y}}
  \tfrac{\delta F}{\delta \pi}(\pi^\eta,x,y)\,
  d(\pi-\pi')(x,y)\,d\eta,
\end{equation*}
where $\pi^\eta=(1-\eta)\pi'+\eta\pi$. Moreover, there exists a constant $C_{\mathrm{osc}}>0$ such that
\[
\sup_{\pi\in\mathcal P(\mathcal X\times\mathcal Y)}
\operatorname{osc}\!\Big(\tfrac{\delta F}{\delta \pi}(\pi,\cdot,\cdot)\Big)
\le C_{\mathrm{osc}},
\]
where for a bounded function $\varphi$ on $\mathcal X\times\mathcal Y$ we denote
$\operatorname{osc}(\varphi)
:= \sup_{(x,y)} \varphi(x,y)
   - \inf_{(x,y)} \varphi(x,y)$.
\item Lipschitz regularity of the derivative: There exists a constant $L>0$ such that, for all $P,Q\in\Pi(\mu,\nu)$,
\[
\Big\|
\tfrac{\delta F}{\delta \pi}(P,\cdot)
-\tfrac{\delta F}{\delta \pi}(Q,\cdot)
\Big\|_{L^\infty(\mu\otimes\nu)}
\;\le\;
L\,W_1(P,Q).
\]
In other words, the functional derivative $\tfrac{\delta F}{\delta \pi}$ is $L$–Lipschitz continuous from $(\Pi(\mu,\nu),W_1)$ into $L^\infty(\mu\otimes\nu)$.
\end{enumerate}
\end{assumption}
Under Assumption~\ref{Ass:F_full}, the functional $F$ is convex, proper, and admits a Lipschitz continuous linear derivative on $(\Pi(\mu,\nu),W_1)$. These regularity properties ensure a well–posed variational structure on the coupling set $\Pi(\mu,\nu)$ as shown in the following proposition.  


\begin{proposition}[Existence and uniqueness of the minimizer of $V$]
\label{prop:exist-uniq-minimizer}
Let Assumptions~\ref{ass:bounded_domain} and \ref{Ass:F_full} hold, and $V(\pi):=F(\pi)+H(\pi\,\|\,R), \pi\in\Pi(\mu,\nu)$. Then there exists a unique minimizer $\pi^\star\in\Pi(\mu,\nu)$ such that
\[
V(\pi^\star)=\inf_{\pi\in\Pi(\mu,\nu)}V(\pi)=:V_\star.
\]
\end{proposition}
The proof of the Proposition is given in Section \ref{sec:main_result_proofs}.

While Proposition \ref{prop:exist-uniq-minimizer} guarantees the existence of a unique entropic coupling $\pi^\star$, the presence of the convex term $F$ destroys the purely multiplicative structure that underlies classical Sinkhorn iterations, so that the minimizer can no longer be computed by a single marginal reweighting. 
To compute the minimizer of the composite objective
\[
\pi \;\longmapsto\; F(\pi)+H(\pi\,\|\,R)
\quad\text{over }\Pi(\mu,\nu),
\]
it is natural to separate the nonlinear convex term $F$ from the strictly convex 
entropic regularization. The latter alone yields the classical entropic optimal transport (EOT) problem, whose minimizer is obtained by Sinkhorn scaling.

\medskip

This motivates a Sinkhorn-Frank-Wolfe (SFW) decomposition, which consists of an \emph{inner} entropic best response and an \emph{outer} Frank--Wolfe update: at each iteration, we linearize $F$ around the current iterate, solve the resulting entropy-plus-linear subproblem by a single Sinkhorn scaling step, and then update the coupling along the corresponding Frank-Wolfe direction.

\textbf{Inner step (Sinkhorn): } Given the current iterate \(P\in\Pi(\mu,\nu)\), we linearize \(F\) at \(P\) and solve 
\[
\min_{\pi\in\Pi(\mu,\nu)}
\Big\langle \tfrac{\delta F}{\delta\pi}(P),\,\pi \Big\rangle + H(\pi\|R).
\]
Under Assumption~\ref{Ass:F_full}, this problem admits a unique minimizer \(\widehat P=T(P)\), characterized by the entropic optimality system
\[
\frac{d\widehat P}{dR}(x,y)
=\exp\!\bigl(-f_P(x)-g_P(y)-\tfrac{\delta F}{\delta\pi}(P,x,y)\bigr),
\]
with potentials \((f_P,g_P)\) enforcing the marginals \((\mu,\nu)\). Thus \(T\) defines the \emph{entropic best--response operator}, obtained through a single Sinkhorn scaling applied to an exponentially tilted reference kernel. The precise formulation is given in Proposition~\ref{prop:inner-problem}.

\textbf{Outer step (Frank--Wolfe): } After the inner step, we update $P$ by moving $P$ toward its best response $T(P)$. In continuous time this yields the Frank--Wolfe flow
\begin{equation}
\label{eq:FW-ODE}
\frac{d}{dt}P_t = T(P_t)-P_t, \qquad P_0\in\Pi(\mu,\nu),
\end{equation}
whose solution admits the variation formula
\begin{equation}
\label{eq:FW-variation}
P_t = e^{-t}P_0 + \int_0^t e^{-(t-s)}\,T(P_s)\,ds.
\end{equation}
Since $P_0$ and $T(P_s)$ have fixed marginals, the trajectory remains in
$\Pi(\mu,\nu)$ for all $t\ge 0$; see Proposition~\ref{prop:FW-feasibility}. So that each iteration consists of one linearization of $F$ followed by one entropic minimization. Under Assumption~\ref{Ass:F_full}, the operator $T$ is Lipschitz on $(\Pi(\mu,\nu),W_1)$, ensuring the well–posedness of the continuous flow and its discrete analogue. 


\begin{proposition}[SFW Inner Problem]
\label{prop:inner-problem}
Suppose Assumption~\ref{Ass:F_full} holds. 
For any fixed $P\in\Pi(\mu,\nu)$, consider the minimization problem
\begin{equation}
\label{eq:inner-problem}
    \inf_{\pi \in \Pi(\mu,\nu)} 
    \int_{\mathcal{X} \times \mathcal{Y}} 
        \tfrac{\delta F}{\delta m}(P,x,y)\,\pi(dx,dy) 
    + H(\pi \,\|\, R).
\end{equation}
where $R$ is a given reference measure on $\mathcal{X}\times\mathcal{Y}$. Then the problem \eqref{eq:inner-problem} admits a unique minimizer, denoted by $\hat{P}$. Moreover, $\hat{P}$ satisfies the first-order optimality condition
\begin{equation}
\label{eq:inner-first-order}
\tfrac{\delta F}{\delta m}(P,x,y)\;+\;f^P(x)\;+\;g^P(y)\;+\;\ln\!\Bigl(\tfrac{\hat{P}(x,y)}{R(x,y)}\Bigr)
\;=\;c_P
\quad\text{for $R$-a.e.\ }(x,y),
\end{equation}
together with the marginal constraints $\hat{P}\in\Pi(\mu,\nu)$. 
\end{proposition}

Proposition~\ref{prop:inner-problem} follows directly from the classical theory of entropic optimal transport. For a fixed $P\in\Pi(\mu,\nu)$, the linear contribution $\pi\mapsto\langle\,\tfrac{\delta F}{\delta m}(P,\cdot),\pi\rangle$ can be absorbed into the reference measure by an exponential tilting.  As recalled in Lemma~\ref{lem:R_P} (see Section~\ref{sec:main_result_proofs} for a detailed proof), adding such a linear term to the entropic objective is, up to an additive constant, equivalent to replacing $R$ by the exponentially tilted kernel $R_{P} \;\propto\; e^{-\tfrac{\delta F}{\delta m}(P,\cdot)}\,R$.
Hence the inner problem \eqref{eq:inner-problem} reduces to the pure entropic minimization
\[
\min_{\pi\in\Pi(\mu,\nu)} H(\pi\|R_P).
\]
The objective is strictly convex in $\pi$, and existence follows from lower semicontinuity together with the compactness of $\Pi(\mu,\nu)$ on bounded domains.The optimality system~\eqref{eq:inner-first-order} is precisely the first order optimality condition for the entropic inner problem with tilted reference measure $R_P$, and determines the dual potentials $(f^P,g^P)$ uniquely up to an additive constant. Such characterizations are classical in entropic optimal transport; see, e.g., Nutz~\cite{nutzEOTnotes}, Carlier et al.~\cite{carlier2017convergence}, and Peyr\'e--Cuturi~\cite{peyre2019}.

\begin{proposition}
\label{prop:wellposedness}
Let $T:\Pi(\mu,\nu)\to\Pi(\mu,\nu)$ be the inner optimal transport map
\begin{equation*}
    T(P) :=\hat{P} \in \text{argmin}_{\pi\in\Pi(\mu,\nu)} \Bigl\{\int_{\mathcal{X}\times\mathcal{Y}} \tfrac{\delta F}{\delta m}(P,x,y)\,\pi(dx,dy) + H(\pi\,\|\,R) \Bigr\},
\end{equation*}
where $\hat{P}$ is defined in Proposition~\ref{prop:inner-problem}.  
Assume that Assumptions~\ref{ass:bounded_domain} and \ref{Ass:F_full} hold. Then there exists a constant $C>0$ such that, for all $P,Q\in\Pi(\mu,\nu)$,
\begin{equation}
\label{eq:T-Lipschitz}
W_1\bigl(T(P),T(Q)\bigr) \;\le\;
C\,\Big\| \tfrac{\delta F}{\delta m}(P,\cdot) -\tfrac{\delta F}{\delta m}(Q,\cdot) \Big\|_{L^\infty(\mu\otimes\nu)}
\;\le\; C\,L\,W_1(P,Q).
\end{equation}
In particular, the map $T$ is $W_1$–Lipschitz on $\Pi(\mu,\nu)$.
\end{proposition}
\begin{proof}
The first inequality of \eqref{eq:T-Lipschitz} follows from Lemma \ref{lem:bounded-transport} and   Proposition~3.12 of Eckstein–Nutz~\cite{eckstein2022quantitative}, and the second one from Assumption~\ref{Ass:F_full}. 
\end{proof}



By Proposition~\ref{prop:wellposedness}, the inner best–response map is well posed: for every feasible coupling \(P\), the problem admits a unique optimizer \(T(P)\in\Pi(\mu,\nu)\), and the dependence \(P\mapsto T(P)\) is quantitatively stable under \(W_1\)–perturbations. This stability will serve as the sole structural input for the outer construction. Before introducing it, we fix an admissible starting point:

\begin{assumption}\label{ass:init-feasible}
The algorithm is initialized with a feasible coupling
\[
P_0 \in \Pi(\mu,\nu),
\]
that is, a probability measure on $\mathcal{X}\times\mathcal{Y}$ whose marginals coincide with $\mu$ and $\nu$. Moreover, $P_0\ll R$ and its density $\rho_0:=dP_0/dR$ is essentially bounded and bounded away from zero: there exist constants $0<b_-\le b_+<\infty$ such that
\[
b_- \le \rho_0 \le b_+
\quad R\text{-a.e.}
\]
This ensures that the marginal constraints are satisfied at the initial step and that all subsequent updates are well defined.
\end{assumption}

With the best--response operator \(T\) well defined on \(\Pi(\mu,\nu)\) and the iteration initialized at a feasible coupling \(P_0\), we now examine the continuous-time evolution obtained by driving \(P\) toward its entropic best response.  This leads to the Frank--Wolfe flow studied below, whose feasibility and structural properties are summarized in the next proposition.


\begin{proposition}
\label{prop:FW-feasibility}
Under Assumptions~\ref{Ass:F_full} and~\ref{ass:init-feasible}, the flow of measures $(P_t)_{t\ge 0}$ defined by \eqref{eq:FW-variation} satisfies
\[
P_t \in \Pi(\mu,\nu), \qquad \text{for all } t\ge 0.
\]
\end{proposition}

See Section~\ref{sec:main_result_proofs} for the proof.  By Proposition~\ref{prop:wellposedness}, the mapping $T$ is $W_1$–Lipschitz on $\Pi(\mu,\nu)$, and therefore \eqref{eq:FW-ODE} has a unique global solution $P_t\in\Pi(\mu,\nu)$, represented by~\eqref{eq:FW-variation}.

The functional $V(\pi)=F(\pi)+H(\pi\,\|\,R)$ serves as the energy associated with the Sinkhorn-Frank-Wolfe flow $(P_t)_{t\ge0}$. Its evolution along the dynamics \eqref{eq:FW-ODE} is characterized by the following energy dissipation identity.

\begin{theorem}[Energy dissipation identity]
\label{thm:EDI}
Let Assumptions~\ref{ass:bounded_domain}, \ref{Ass:F_full}, and \ref{ass:init-feasible} hold, and let $(P_t)_{t\ge0}$ be the Sinkhorn-Frank-Wolfe (SFW) flow defined by \eqref{eq:FW-ODE}. Set
\[
V(\pi):=F(\pi)+H(\pi\,\|\,R),
\qquad 
V_\star := \inf_{\pi\in\Pi(\mu,\nu)} V(\pi).
\]
Then for every $t\ge0$, the map $t\mapsto V(P_t)$ is differentiable and
\begin{equation}
\label{eq:EDI-eq}
\frac{d}{dt}V(P_t) 
= -\Bigl(H(\hat{P}_t\,\|\,P_t)+H(P_t\,\|\,\hat{P}_t)\Bigr)
\;\le\; -\bigl(V(P_t)-V_\star\bigr),
\end{equation}
where $\hat{P}_t$ is the inner minimizer given by Proposition~\ref{prop:inner-problem}. In particular, $ V( P_t ) - V_* \leq e^{-t} ( V( P_0 ) - V_* ) $ and hence $V(P_t)\downarrow V_*$ exponentially fast.
\end{theorem}

The energy dissipation identity of Theorem~\ref{thm:EDI} yields a strict dissipation mechanism for the functional $V$ along the SFW flow.  Similar dissipation has been discovered in the context of mean-field minimization in \cite{chen2023entropic}.   This structural estimate allows us to identify the minimizer of $V$ and to derive quantitative convergence of the dynamics.



\begin{theorem}[Exponential convergence of the SFW flow]
\label{thm:exp-conv}
Suppose that Assumptions~\ref{ass:bounded_domain}, \ref{Ass:F_full}, and~\ref{ass:init-feasible} hold, and let $\pi^\star\in\Pi(\mu,\nu)$ denote the unique minimizer of $V$ given by Proposition \ref{prop:exist-uniq-minimizer}. Let $(P_t)_{t\ge0}$ be the Sinkhorn-Frank-Wolfe flow with initial condition $P_0$ satisfying Assumption~\ref{ass:init-feasible}. Then, for all $t\ge0$,
\begin{equation}
\label{eq:exp-energy}
H(P_t\,\|\,\pi^\star) 
\;\le\;
V(P_t)-V(\pi^\star) 
\;\le\;
e^{-t}\bigl(V(P_0)-V(\pi^\star)\bigr),
\qquad t\ge0.
\end{equation}
In particular, $V(P_t)\downarrow V(\pi^\star)$ at an exponential rate, and $P_t$ converges to $\pi^\star$ in relative entropy, i.e.\ $H(P_t\,\|\,\pi^\star)\to 0$, hence also in total variation and in the weak topology.
\end{theorem}

Theorems~\ref{thm:EDI} and~\ref{thm:exp-conv} together reveal the complete variational structure of the Sinkhorn-Frank-Wolfe flow. The energy dissipation identity identifies the symmetric relative entropy $H(\hat{P}_t\,\|\,P_t)+H(P_t\,\|\,\hat{P}_t)$ as the coercive quantity governing the decay of the composite functional $V$ along the dynamics. This dissipation mechanism, combined with the convexity properties encoded in Assumption~\ref{Ass:F_full}, yields both the existence of the unique minimizer $\pi^\star$ of $V$ and the exponential convergence of the trajectory $(P_t)_{t\ge0}$ toward $\pi^\star$.  

In particular, the estimate \eqref{eq:exp-energy} implies that $P_t$ converges to $\pi^\star$ in relative entropy, and hence also in total variation and in the weak topology. The proofs of Theorems~\ref{thm:EDI} and~\ref{thm:exp-conv} are provided in  Section~\ref{sec:main_result_proofs}. In Section~\ref{sec:UAV}, we complement the theoretical analysis with numerical experiments arising from a drone--relocation framework, viewed as a concrete instance of our convex cost setting, in which the local density of UAVs enters the objective through a nonlinear congestion-type penalty. These simulations illustrate how the structural constants in our analytical bounds influence both the theoretical dissipation rate and the observed convergence behavior of the Sinkhorn-Frank-Wolfe flow.


\section{Application to UAV Relocation}
\label{sec:UAV}
\subsection{Motivation}

Unmanned aerial vehicles (UAVs) are increasingly deployed in large autonomous fleets. In large-scale UAV operations, returning drones must be redistributed across a set of stations or service regions. Since individual vehicles are interchangeable, their relocation can be described at a macroscopic level by spatio-temporal probability measures. A key operational challenge is congestion: excessive flow through narrow regions of space or time leads to safety and capacity constraints. These effects can be modelled by a nonlinear congestion-type penalty acting on the transport plan through its local density.

This setting naturally leads to an entropic optimal transport problem with convex functional cost, where a reference kernel encodes nominal movement tendencies and a convex congestion functional $F$ encodes local density costs. The Sinkhorn-Frank-Wolfe (SFW) flow developed in Section~\ref{sec:main_result} applies directly in this framework, providing a scalable method for computing congestion-aware relocation plans.

\subsection{UAV Relocation Framework}
\label{subsec:UAV_model}
We work on two bounded tempo–spatial domains $\mathcal{X,Y} \subset \mathbb{R}_{+} \times \mathbb{R}^{2}$, and denote points as $(t,x)\in\mathcal{X}$ for initial time–space states and $(s,y)\in\mathcal{Y}$ for target ones. The current and desired UAV mass distributions are denoted by $\mu \in \mathcal P(\mathcal{X})$, $\nu \in \mathcal P(\mathcal{Y})$ and a feasible relocation strategy is a coupling $\pi\in\Pi(\mu,\nu)\subset\mathcal P(\mathcal{X \times Y})$ assigning mass from $(t,x)$ to $(s,y)$.

In contrast to classical optimal transport, UAV relocation must additionally account for congestion effects: routing excessive mass through narrow tempo–spatial regions increases collision risk and operational costs. We encode these effects through a convex congestion functional $F$ acting on $\pi$, together with a reference kernel $R$ on $\mathcal{X \times Y}$ describing nominal movement tendencies from $(t,x)$ to $(s,y)$. The relocation problem, therefore, fits into the entropy–regularized convex transport framework of Section~\ref{sec:main_result} via the objective
\[
V(\pi)=F(\pi)+H(\pi\,\|\,R),
\qquad \pi\in\Pi(\mu,\nu).
\]

The Sinkhorn–Frank–Wolfe flow developed earlier applies directly to this setting and provides a scalable, distributional algorithm for computing the optimal relocation plan. We next specify the choices of the reference kernel $R$, the transport cost $c(t,x,s,y)$, and the congestion functional $F$ relevant to UAV operations.

If we take the reference measure $R$ of the form $dR \propto e^{-c/\varepsilon}\,d(\mu\otimes\nu)$, then the objective function of our entropy–regularized transport problem with congestion aversion can be written as
\[
V(\pi) = 
\int_{\mathcal{X \times Y}} c(t,x,s,y)\,d\pi
+ F(\pi) + \epsilon H(\pi \|\mu \otimes \nu),
\qquad \pi\in\Pi(\mu,\nu),
\]

The first term represents the physical transportation cost determined by the energy model $c$, while the congestion functional $F$ penalizes excessive mass concentration in tempo–spatial bottlenecks.  This places the UAV relocation problem fully within the entropy–regularized convex transport framework developed in Section~\ref{sec:main_result}.  We next specify the transport cost $c$ and the congestion penalty $F$ relevant to UAV operations.

\textbf{Transport cost:} The physical energy cost function $c(t, x, s, y)$ quantifies the energy expenditure of UAVs during  flights, which incorporates both spatial and temporal factors:
\begin{equation}
\label{eq:primal_cost}
    c(t, x, s, y) = \inf_{r \leq s - t} \left(\lambda \cdot \frac{d(x,y)^2}{r^2} + \beta \right) \cdot r \cdot \mathbf{1}_{s \geq t} + \infty \cdot \mathbf{1}_{s < t}, 
\end{equation}
where $r$ is the travel time, a free variable constrained by $0\le r \leq s - t$ and $d(x,y)$ is the Euclidean distance. This cost function comprises two key components: the first term $\lambda \frac{d(x,y)^2}{r^2}$ represents the kinetic energy expenditure. This reflects the fundamental trade-off in UAV motion, where minimizing travel time incurs a higher energy cost due to the increased velocity requirement. The second term, $\beta \cdot r$, accounts for the baseline operational energy consumption over the travel duration $r$. This term captures the energy required to maintain flight, independent of velocity, and provides a lower bound on the total energy expenditure. The constraint $\infty \cdot \mathbf{1}_{s < t}$ prevents physically infeasible scenarios where a UAV would arrive before its departure time. 

Minimizing $c(t, x, s, y)$ with respect to $r$ leads to the optimal travel time in \eqref{eq:primal_cost}:
\begin{equation}
\label{eq:r_min}
    r_\text{min} = \sqrt{\frac{\lambda \cdot d(x,y)^2}{\beta}}.
\end{equation}
Thus, the transport cost function takes the form:
\begin{equation*}
    \begin{aligned}
    \label{eq:costfunction}
        c(t, x, s, y) = 
            \begin{cases}
                \left(\lambda \cdot \frac{d(x,y)^2}{r_\text{min}^2} + \beta\right) \cdot r_\text{min}, & r_\text{min} \leq s - t, \\
                \left(\lambda \cdot \frac{d(x,y)^2}{(s - t)^2} + \beta\right) \cdot (s - t), & r_\text{min} > s - t.
            \end{cases}
    \end{aligned}
\end{equation*}
This construction yields a continuous transport cost satisfying the convexity
requirements of Assumption~\ref{Ass:F_full}.

\textbf{Density penalization and congestion control: }
In tempo-spatial UAV relocation, localized concentration of trajectories creates safety and regulatory risks. To capture such effects at the distributional level, we introduce a convex congestion functional $F(\pi)$ that penalizes transport plans generating excessive density along UAV paths. Given a transport plan $\pi$, the UAV trajectory from $(t,x)$ to $(s,y)$ follows a deterministic path determined by its velocity. The minimum travel time $r_{\min}$ \eqref{eq:r_min} dictates the UAV’s velocity, allowing us to express its location at any intermediate time $s'$ as
\begin{equation}
\label{eq:trajectory}
    T_{(t,x,s,y)}(s') \;=\; x + \frac{y-x}{r_{\min}}(s'-t).
\end{equation}
Thus, each UAV’s trajectory is fully determined by its initial and final states. 

Take a tempo-spatial zone $\mathcal{Z}$ such that $\left\{\big(s', T_{(t,x,s,y)}(s')\big):~ t\le s'\le r_{\rm min}\right\} \subset \mathcal{Z}$. Let \(\{A_n\}_{n=1,\cdots,N}\) be a partition of $\mathcal{Z}$, that is, 
\begin{align*}
    \mathcal{Z} = \bigcup_{n=1}^N A_n, \quad A_n \cap A_m = \emptyset, \quad \forall n \neq m.
\end{align*}
 The density of UAVs located within a given tempo-spatial subregion $A_n$ is determined by \eqref{eq:trajectory} as
\begin{equation*}
    \int_{\mathcal{X \times Y}}
        \max_{s\le s'\le r_{\min}}
        \mathbf 1_{A_n}\bigl(s',T_{(t,x,s,y)}(s')\bigr)\, d\pi,
\end{equation*}
 This integral represents the total number of UAVs in $(A_n$ at time $s'$. In the physical sense, if this number is too large, the local density is considered excessive. The overall congestion penalty is then defined by summing over all subregions \(A_n\):
\begin{equation*}
    F(\pi)
    \;=\;
    \sum_{n=1}^N
        f_p\!\left(
            \int_{\mathcal{X \times Y}}
            \max_{s\le s'\le r_{\min}}
            \mathbf 1_{A_n}\bigl(s',T_{(t,x,s,y)}(s')\bigr)\, d\pi
        \right),
\end{equation*}
where $f_p(\cdot)$ is a convex function that imposes a higher cost for larger UAV density. In this way, the entire penalty term, as a function of the optimal transport plan $\pi$, encourages the redistribution of UAV trajectories by adjusting the final destinations of UAVs, thereby alleviating congestion. Since $f_p$ is convex, the functional $F$ is convex in $\pi$. Moreover, its linear derivative with respect to $\pi$ takes the explicit form
\begin{equation*}
\label{eq:dF/dpi}
\begin{aligned}
    \frac{\delta F}{\delta \pi}(\pi,t,x,s,y) =
    \sum_{n=1}^N
        \nabla f_p (
            \int_{\mathcal{X \times Y}} \max_{s\le s'\le r_{\min}}
            & \mathbf{1}_{A_n}\bigl(s',T_{(t,x,s,y)}(s')\bigr)\, d\pi ) \\
        & \cdot
        \max_{s\le s'\le r_{\min}}
        \mathbf{1}_{A_n}\bigl(s',T_{(t,x,s,y)}(s')\bigr).
\end{aligned}
\end{equation*}

\subsection{Numerical Experiments}
For the numerical experiments, we implement the discrete Sinkhorn-Frank-Wolfe 
scheme described in Algorithm~\ref{algo:SFW}. The state spaces $\mathcal{X,Y}$ are taken as finite grids in the tempo--spatial domain $[0,T]\times[0,1]^2$, namely. More precisely, $ \mathcal{X} = \{ ( t_k, x_i, y_j ), \, k = 1, \ldots, K, \, i = 0, \ldots, n-1, \, j = 0, \ldots, m-1 \} $ and $ \mathcal{Y} = \{ ( s_l, x_i, y_j ), \, l = 1, \ldots, L, \, i = 0, \ldots, n-1, \, j = 0, \ldots, m-1 \} $ with $ t_k, s_l \in [0, T] $, $ x_i = i / n $, $ y_j = j/m $.

\begin{algorithm}[htbp]
\caption{Sinkhorn-Frank-Wolfe (SFW) Algorithm}
\label{algo:SFW}
\begin{algorithmic}[1]
\Require initial feasible coupling $P_{0}\in\Pi(\mu,\nu)$ with $P_{0}\ll R$;
update rate $\alpha \ge 0$; maximal number of
outer iterationss $N_{\max}^O$; maximal number of 
inner iterations $N_{\max}^I$; 
tolerance $tol$
\Ensure approximate minimizer $P_{N}$ of $V$

\For{$n=0,1,\dots,N^O_{\max}-1$}

    \State \textbf{Outer Frank--Wolfe step:} compute
    \[
        \phi^{n}(t,x,s,y)
        :=\tfrac{\delta F}{\delta m}\bigl(P_{n},t,x,s,y\bigr).
    \]

    \State \hspace{1.2em} \textbf{Inner Sinkhorn iteration:} initialize $f^{n,0}\equiv0$, $g^{n,0}\equiv0$, $\widetilde P_{\,n,0}:=P_n$.

    \State \hspace{1.5em}\textbf{for } $k = 0,1,\dots,N_{\max}^I-1$

    \State \hspace{3em} Form Gibbs update:
    \[
    \hspace{3em}
    \widetilde P_{n,k+1}(t,x,s,y)
    \propto
    \exp\Bigl(
        -\phi^n(t,x,s,y)
        -f^{n,k}(t,x)
        -g^{n,k}(s,y)
    \Bigr)dR(t,x,s,y).
    \]

    \State \hspace{3em} Update $f$ to enforce the $\mu$--marginal:
    \[
    \hspace{4em}
    f^{n,k+1}(t,x)
    \gets
    f^{n,k}(t,x)
    +
    \log\!\Biggl(
        \frac{d\mu(t,x)}
        {\displaystyle\int_{\Omega}\!
                \widetilde P_{\,n,k+1}(t,x,s,y)\,ds\,dy}
    \Biggr).
    \]

    \State \hspace{3em} Update $g$ to enforce the $\nu$--marginal:
    \[
    \hspace{4em}
    g^{n,k+1}(s,y)
    \gets
    g^{n,k}(s,y)
    +
    \log\!\Biggl(
        \frac{d\nu(s,y)}
        {\displaystyle\int_{\Omega}\!
                \widetilde P_{\,n,k+1}(t,x,s,y)\,dt\,dx}
    \Biggr).
    \]

    \State \hspace{3em}\textbf{if }
    $\|\widetilde P_{\,n,k+1}-\widetilde P_{\,n,k}\|_{\mathrm{TV}}\le tol$
    \textbf{ then break}

    \State \hspace{1.5em}\textbf{end for}

    \State \hspace{1.2em} Set inner minimizer $\hat P_{\,n} := \widetilde P_{\,n,k+1}$.

    \State \textbf{Outer FW update:}
    \[
        P_{n+1} = (1-\alpha) P_n + \alpha \hat P_{\,n}.
    \]

    \If{$\|P_{n+1}-P_n\|_{\mathrm{TV}} \le tol$}
        \State \textbf{break}
    \EndIf

\EndFor

\State \Return $P_{n+1}$
\end{algorithmic}
\end{algorithm}

We next summarize the parameter choices used in the simulations. For the tempo--spatial grid we take $T=0.5$, $n=60$, and $m=41$. For the optimization model, we use the entropic parameter $\epsilon=0.1$ and a quadratic congestion penalty $f_p(z)=\gamma |z|^2$ with $\gamma=20$. The transport cost employs $\lambda=1$ and $\beta=0.001$. The SFW iterations use some fixed update rates $\alpha=0.01, 0.02, 0.04$, with at most $N^O_{\max}=40$ outer updates and $N^I_{\max}=30$ inner Sinkhorn iterations, and a stopping tolerance $\text{tol}=10^{-4}$.

Figures~\ref{fig:no_cong} and~\ref{fig:cong} illustrate the impact of the congestion penalization on the UAV trajectories. We consider $K=L=1$, with the initial state modelled as a sum of half-Gaussians and the target state concentrated at two discrete points. Without penalization, the UAVs naturally move toward the nearest target (Figure~\ref{fig:no_cong}). When the congestion term is introduced (Figure~\ref{fig:cong}), the algorithm mitigates congestion near the targets by redirecting part of the flow toward the farther location, hence visibly altering the transport directions.


Figure~\ref{fig:Vgap} displays the evolution of the energy gap $V(P_t) - V_\ast$ along the Sinkhorn--Frank--Wolfe iterations, where the horizontal axis is rescaled to the continuous time $t = \alpha s$ associated with the Frank--Wolfe flow. In semi-logarithmic scale, the three curves corresponding to $\alpha = 0.01, 0.02, 0.04$ appear essentially linear and almost collapse onto a single trajectory once expressed in time units, in agreement with the exponential dissipation rate established in Theorem~\ref{thm:exp-conv}. This indicates that the discrete iterations faithfully approximate the underlying continuous-time SFW flow and inherit its contraction mechanism.

A comparison between the three values of $\alpha$ further reveals that, once expressed in the time variable $t$, the trajectories nearly coincide.  Although different step sizes correspond to different discretizations of the flow, the resulting curves follow essentially the same decay profile.  This demonstrates that the discrete Sinkhorn-Frank-Wolfe method is a faithful time discretization of the continuous Frank--Wolfe flow: larger $\alpha$ simply progress faster along 
the same trajectory, whereas smaller $\alpha$ produce a finer sampling of the same underlying evolution.

\begin{figure}[htbp]
    \centering
    \includegraphics[width=\textwidth]{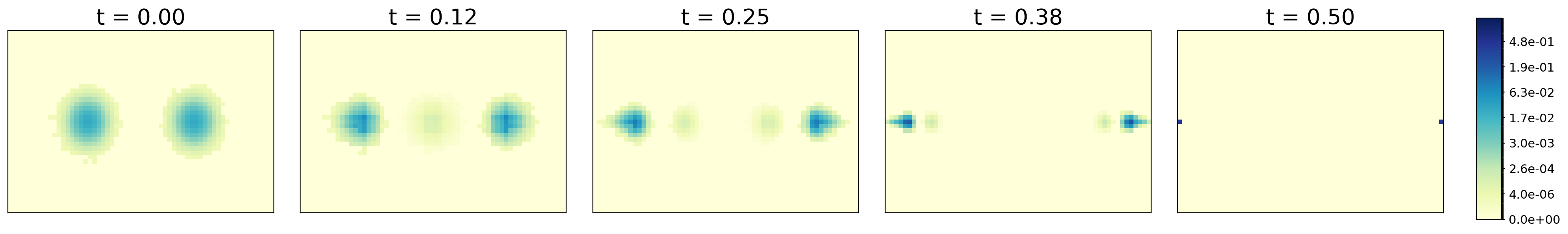}
    \caption{Sinkhorn algorithm without congestion penalization.}
    \label{fig:no_cong}
\end{figure}

\begin{figure}[htbp]
\centering
\includegraphics[width=\textwidth]{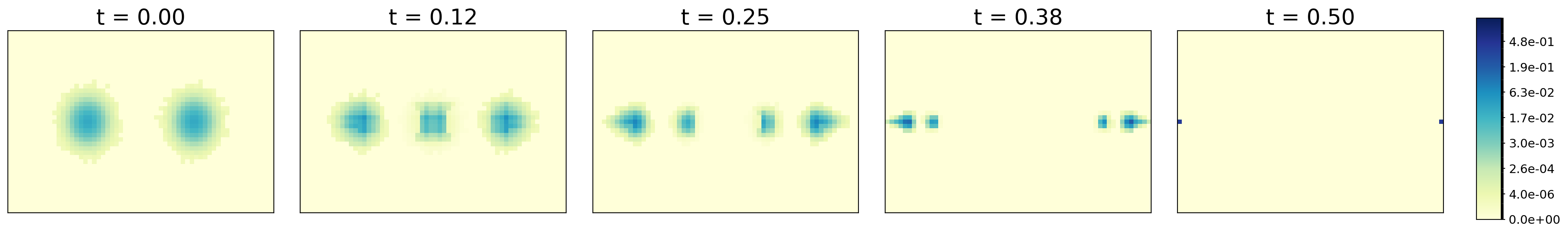}
\caption{Sinkhorn-Frank-Wolfe algorithm with congestion penalization.}
\label{fig:cong}
\end{figure}

\begin{figure}[htbp]
\centering
\includegraphics[width=0.6\textwidth]{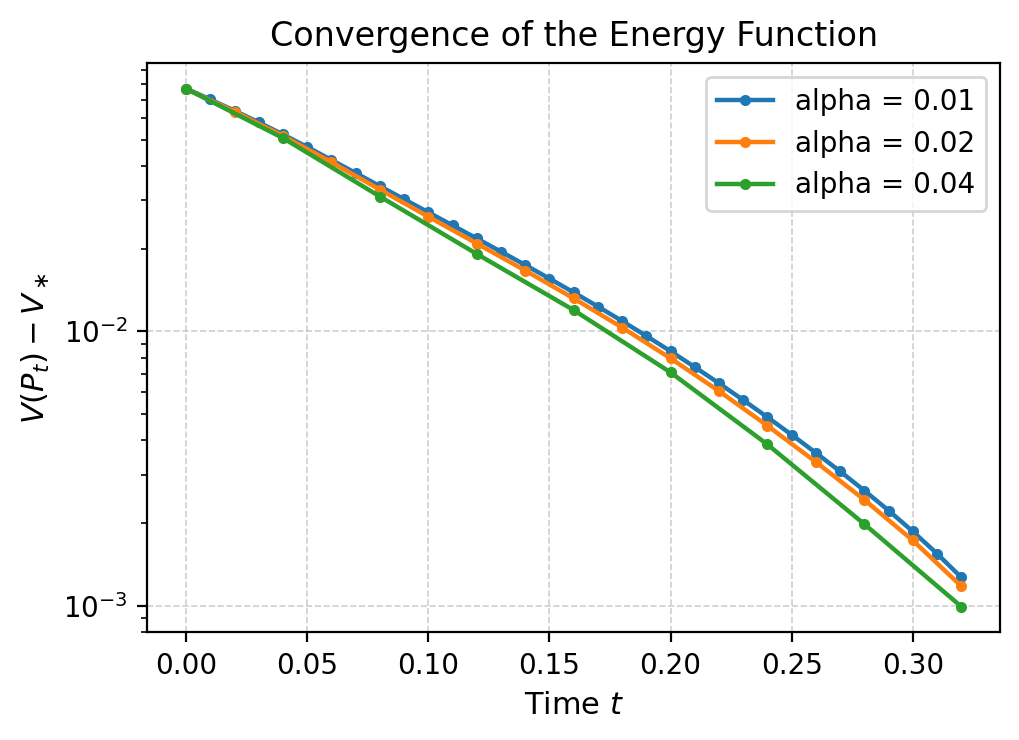} 
\caption{Evolution of the energy gap $V(P_t)-V_\ast$ along the Sinkhorn--Frank--Wolfe iterations, with the horizontal axis rescaled to the continuous time $t = \alpha s$. The three curves correspond to step sizes $\alpha = 0.01, 0.02, 0.04$ and, once expressed in time units, exhibit essentially the same decay profile, indicating that the convergence rate is intrinsic to the flow rather than to the discretization parameter $\alpha$.}
\label{fig:Vgap}
\end{figure}

\section{Proof of the Main Results}
\label{sec:main_result_proofs}

\subsection{Proof of Lemmas and Propositions}
\begin{lemma}[Transport--entropy bound on bounded domains]
\label{lem:bounded-transport}
Let Assumption~\ref{ass:bounded_domain} hold, and let 
$\pi,\theta$ be probability measures supported on 
$\mathcal{X}\times\mathcal{Y}$. 
Then, for every $q\in[1,\infty)$,
\begin{equation}\label{eq:bounded-transport}
W_q(\pi,\theta)
\;\le\;
C_q\,H(\pi\,|\,\theta)^{1/(2q)},
\qquad
C_q
=\mathrm{diam}(\mathcal{X}\times\mathcal{Y})^{\,1-\frac1q}\,2^{-1/(2q)}.
\end{equation}
\end{lemma}

\begin{proof}[Proof of Lemma \ref{lem:bounded-transport}]
For any $x,y\in\mathcal{X}\times\mathcal{Y}$ and $q\ge1$,
\[
\|x-y\|^q
\le
\mathrm{diam}(\mathcal{X}\times\mathcal{Y})^{\,q-1}\|x-y\|.
\]
Integrating with respect to an optimal coupling between $\pi$ and $\theta$
yields the equivalence of Wasserstein metrics on bounded domains:
\[
W_q(\pi,\theta)^q
\le 
\mathrm{diam}(\mathcal{X}\times\mathcal{Y})^{\,q-1}W_1(\pi,\theta).
\]
Applying the classical Pinsker inequality $W_1(\pi,\theta)\le \sqrt{\tfrac12\,H(\pi|\theta)}$,
we obtain
\[
W_q(\pi,\theta)
\le 
\mathrm{diam}(\mathcal{X}\times\mathcal{Y})^{\,1-\frac1q}
\bigl(\tfrac12\bigr)^{1/(2q)}
H(\pi|\theta)^{1/(2q)}.
\]
This proves \eqref{eq:bounded-transport}.
\end{proof}

\begin{remark}
    Lemma~\ref{lem:bounded-transport} provides the transport--entropy estimate \eqref{eq:bounded-transport}, which ensures that pairs of measures on bounded domains satisfy the $(I_q)$ transport--entropy condition of Eckstein--Nutz~\cite[Proposition~3.12]{eckstein2022quantitative}.  This property will be used in the proof of Proposition~\ref{prop:wellposedness} to control the Wasserstein stability of the inner best--response operator.
\end{remark}

\begin{proof}[Proof of Proposition \ref{prop:exist-uniq-minimizer}]
By Prokhorov's theorem, the feasible set $\Pi(\mu,\nu)$ is tight and closed in 
the weak topology on $\mathcal P(\mathcal X\times\mathcal Y)$, hence it is 
weakly compact. Assumption~\ref{Ass:F_full} ensures that the functional $F$ is 
bounded from below and weakly lower semicontinuous on $\Pi(\mu,\nu)$, and the 
relative entropy $H(\cdot\,\|\,R)$ is also weakly lower semicontinuous. Thus
\[
V(\pi):=F(\pi)+H(\pi\,\|\,R)
\]
is weakly lower semicontinuous and bounded from below on $\Pi(\mu,\nu)$.
Let $(\eta^n)_{n\ge1}\subset\Pi(\mu,\nu)$ be a minimizing sequence, i.e.
\[
V(\eta^n)\downarrow 
V_\star:=\inf_{\pi\in\Pi(\mu,\nu)}V(\pi)>-\infty.
\]
Since $F$ is bounded from below by some constant $m\in\mathbb R$, we have
\[
H(\eta^n\,\|\,R)
\le V(\eta^n)-m
\le V(\eta^1)-m,
\qquad n\ge1.
\]
Hence there exists $C<\infty$ such that  
\[
H(\eta^n\,\|\,R)\le C,\qquad n\ge1.
\]
Define the entropy sublevel set
\[
K_C:=\{\pi\in\Pi(\mu,\nu): H(\pi\,\|\,R)\le C\}.
\]
Since $H(\cdot\,\|\,R)$ is weakly lower semicontinuous, the sublevel set $\{\pi\in\mathcal P(\mathcal X\times\mathcal Y):H(\pi\,\|\,R)\le C\}$ is weakly closed in $\mathcal P(\mathcal X\times\mathcal Y)$. Hence $K_C$ is weakly closed in $\Pi(\mu,\nu)$, and therefore weakly compact as a closed subset of the weakly compact set $\Pi(\mu,\nu)$. Moreover, $H(\cdot\,\|\,R)$ is convex, so $K_C$ is a convex subset of $\Pi(\mu,\nu)$. Since $(\eta^n)$ is contained in $K_C$, compactness yields a weakly convergent subsequence: there exist indices $n_k\uparrow\infty$ and some $\bar\pi\in K_C$ such that
\[
\eta^{n_k} \rightharpoonup \bar\pi
\quad \text{in } \mathcal P(\mathcal X\times\mathcal Y).
\]
By weak lower semicontinuity of $V$,
\[
V(\bar\pi)
\;\le\; \liminf_{k\to\infty}V(\eta^{\,n_k})
\;=\; V_\star.
\]
Hence $V(\bar\pi)=V_\star$ and $\bar\pi\in K_C$ is a minimizer of $V$. 

Since $H(\bar\pi\,\|\,R)\le C<\infty$, necessarily $\bar\pi\ll R$. To prove uniqueness, note that $H(\cdot\,\|\,R)$ is strictly convex on the set
$\{\pi:\pi\ll R\}$, while $F$ is convex by Assumption~\ref{Ass:F_full}.  
Thus $V=F+H(\cdot\,\|\,R)$ is strictly convex on $K_C$.  
Any minimizer of $V$ lies in $K_C$ and is absolutely continuous with respect
to $R$, hence no two distinct minimizers can exist.  
Therefore the minimizer is unique; we denote it by $\pi^\star$, and
\[
V_\star=\inf_{\pi\in\Pi(\mu,\nu)}V(\pi)=V(\pi^\star).
\]
\end{proof}

\begin{lemma}[Exponential tilting and absorption of linear terms]
\label{lem:R_P}
Let $R\in\mathcal P(\mathcal X\times\mathcal Y)$ be a fixed reference measure with $R\ll\mu\otimes\nu$, and let $\varphi:\mathcal X\times\mathcal Y\to\mathbb R$ be any bounded measurable function. Define the exponentially tilted probability measure $R_\varphi\in\mathcal P(\mathcal X\times\mathcal Y)$ by
\begin{equation}
\label{eq:def-Rphi}
dR_\varphi(x,y)
\;:=\;
Z_\varphi^{-1}\,e^{-\varphi(x,y)}\,dR(x,y),
\qquad 
Z_\varphi:=\int_{\mathcal{X\times Y}} e^{-\varphi}\,dR\in(0,\infty).
\end{equation}
Then, for every $\pi\in\Pi(\mu,\nu)$ with $\pi\ll R$, the following identity holds:
\begin{equation}\label{eq:H-tilting}
\int_{\mathcal X\times\mathcal Y}\!\varphi(x,y)\,d\pi(x,y)
+H(\pi\,\|\,R)
\;=\;
H(\pi\,\|\,R_\varphi)+\log Z_\varphi.
\end{equation}
Consequently, minimizing the left-hand side over $\pi\in\Pi(\mu,\nu)$ is equivalent, up to an additive constant, to minimizing the pure entropic term $H(\pi\,\|\,R_\varphi)$.
\end{lemma}
\begin{proof}[Proof of Lemma~\ref{lem:R_P}]
\label{pro:lemma_R_P}
The result follows from the Radon–Nikodym chain rule:
\[
\log\!\Bigl(\tfrac{d\pi}{dR}\Bigr)
=
\log\!\Bigl(\tfrac{d\pi}{dR_\varphi}\Bigr)
+\log\!\Bigl(\tfrac{dR_\varphi}{dR}\Bigr)
=
\log\!\Bigl(\tfrac{d\pi}{dR_\varphi}\Bigr)
-\varphi - \log Z_\varphi.
\]
Multiplying by $d\pi$ and integrating gives
\[
\int \log\!\Bigl(\tfrac{d\pi}{dR}\Bigr)\,d\pi
=
\int \log\!\Bigl(\tfrac{d\pi}{dR_\varphi}\Bigr)\,d\pi
-\int \varphi\,d\pi
-\log Z_\varphi,
\]
that is,
\(
H(\pi\,\|\,R)
=
H(\pi\,\|\,R_\varphi)-\int\varphi\,d\pi-\log Z_\varphi,
\)
which rearranges to \eqref{eq:H-tilting}.
\end{proof}

\begin{remark}
Lemma~\ref{lem:R_P} is the standard exponential tilting identity used throughout the theory of Schr\"odinger-type problems and entropic optimal transport. We recall it here for completeness, as it underlies the reduction of the linearized subproblem to a pure entropic minimization in Proposition~\ref{prop:inner-problem}.
\end{remark}

\begin{proof}[Proof of Proposition \ref{prop:FW-feasibility}.]
Fix $A\in\mathcal B(\mathcal X)$ and set $m_X(t):=P_t(A\times\mathcal Y)$. Using \eqref{eq:FW-ODE} and linearity,
\[
\dot m_X(t)
=\hat{P}_t(A\times\mathcal Y)-P_t(A\times\mathcal Y)
=\mu(A)-m_X(t),
\]
since $\hat{P}_t\in\Pi(\mu,\nu)$ has $X$-marginal $\mu$. With $m_X(0)=P_0(A\times\mathcal Y)=\mu(A)$ by Assumption \ref{ass:init-feasible}, the unique solution is $m_X(t)\equiv \mu(A)$ for all $t\ge 0$. 
An identical argument for $m_Y(t):=P_t(\mathcal X\times B)$ with $B\in\mathcal B(\mathcal Y)$ yields $m_Y(t)\equiv \nu(B)$.
Hence $P_t\in\Pi(\mu,\nu)$ for all $t\ge 0$.
Alternatively, from \eqref{eq:FW-variation},
\[
P_t(A\times\mathcal Y)
= e^{-t}\mu(A) + \int_0^t e^{-(t-s)}\,\mu(A)\,ds
= \mu(A),
\]
and similarly for the $Y$-marginal.
\end{proof}

\subsection{Proof of Theorem \ref{thm:EDI}}
Before proving Theorem \ref{thm:EDI}, we show a lemma on the uniform integrability of $\hat{P}_t$ and $P_t$.
\begin{lemma}[Uniform envelopes for log-densities]\label{lem:Uniform_integrable}
Assume \ref{ass:bounded_domain}, \ref{Ass:F_full}, and \ref{ass:init-feasible}. Let $(P_t)_{t\ge0}$ be the flow defined by \eqref{eq:FW-ODE}--\eqref{eq:FW-variation}, and let $\hat{P}_t$ be the inner minimizers from Proposition~\ref{prop:inner-problem}, with dual potentials $(f_t,g_t)$ normalized by 
$\int f_t\,d\mu=0$.  
Then there exist constants 
\(G_{\hat\pi},G_{P}<\infty\), independent of \(t\), such that for all \(t\ge 0\),
\[
\Bigl|\log\tfrac{d\hat{P}_t}{dR}\Bigr|\le G_{\hat\pi},
\qquad
\Bigl|\log\tfrac{dP_t}{dR}\Bigr|\le G_{P}
\quad R\text{-a.e.\ on }\mathcal X\times\mathcal Y.
\]
In particular,
\[
\Bigl|\log\tfrac{dP_t}{dR}-\log\tfrac{d\hat{P}_t}{dR}\Bigr|
\le G_{\hat\pi}+G_{P},
\qquad t\ge 0.
\]
\end{lemma}

\begin{proof}
By Assumptions~\ref{ass:bounded_domain} and \ref{Ass:F_full}(iii), there exists 
\(M<\infty\) such that, for all $t\ge0$,
\[
\phi_t(x,y):=\tfrac{\delta F}{\delta m}(P_t,x,y),
\qquad 
\|\phi_t\|_{L^\infty(\mu\otimes\nu)}\le M.
\]
Hence $e^{-M}\le e^{-\phi_t(x,y)}\le e^{M}, R\text{-a.e.,\ for all } t\ge0 $. For each \(t\ge0\), Proposition~\ref{prop:inner-problem} yields
\[
\frac{d\hat{P}_t}{dR}(x,y)
= \exp\bigl(-f_t(x)-g_t(y)-\phi_t(x,y)\bigr),
\]
with marginals \((\hat{P}_t)_X=\mu\), \((\hat{P}_t)_Y=\nu\), and \(\int f_t\,d\mu=0\). Lemma~6.14 of Nutz~\cite{nutzEOTnotes}, applied with cost \(c=\phi_t\), yields constants 
\(C_f,C_g<\infty\), independent of \(t\), such that
\[
\|f_t\|_{L^\infty(\mu)}\le C_f,
\qquad
\|g_t\|_{L^\infty(\nu)}\le C_g,
\qquad t\ge0.
\]
Therefore we have
\begin{equation}
\label{eq:bound_hat_pi}
    \Bigl|\log\tfrac{d\hat{P}_t}{dR}\Bigr|
    = |f_t+g_t+\phi_t| \le C_f + C_g + M
    =: G_{\hat\pi}, \qquad t\ge0.
\end{equation}
We now bound \(dP_t/dR\) on \(t\ge 0\). From \eqref{eq:FW-variation}, $P_t = e^{-t}P_0 + \int_0^t e^{-(t-s)}\,\hat{P}_s\,ds, t\ge0$. Since all measures are absolutely continuous with respect to \(R\), there exist densities
\[
dP_t=\rho_t\,dR,
\qquad
d\hat{P}_s=\hat\rho_s\,dR.
\]
From the bound \eqref{eq:bound_hat_pi}, we obtain constants \(0<C_-\le C_+<\infty\) such that
\[
C_- \le \hat\rho_s \le C_+,
\qquad R\text{-a.e.,\ for all }s\ge0.
\]
Note that
\[
\rho_t 
= e^{-t}\rho_0 + \int_0^t e^{-(t-s)}\,\hat\rho_s\,ds.
\]
By Assumption~\ref{ass:init-feasible}, \(b_-\le\rho_0\le b_+\). Using this property and \(C_-\le\hat\rho_s\le C_+\) and 
\(\int_0^t e^{-(t-s)}ds = 1-e^{-t}\), we obtain the pointwise bounds
\[
\rho_t 
\ge e^{-t} b_- + C_-(1-e^{-t}),
\qquad
\rho_t 
\le e^{-t} b_+ + C_+(1-e^{-t}),
\qquad t\ge 0.
\]
Since \(e^{-t}\in(0,1]\), each \(\rho_t\) is a convex combination of \(b_-\) and \(C_-\), and of \(b_+\) and \(C_+\) from above. Hence there exist constants
\[
a_- := \min\{b_-,C_-\}>0,
\qquad
a_+ := \max\{b_+,C_+\}<\infty,
\]
such that
\[
a_- \le \rho_t \le a_+,
\qquad R\text{-a.e.,\ for all }t\ge 0.
\]
It follows that
\[
\Bigl|\log\tfrac{dP_t}{dR}\Bigr|
\le \max\{|\log a_-|,|\log a_+|\}
=: G_{P},
\qquad  t \ge 0.
\]
Combining the bounds for \(P_t\) and \(\hat{P}_t\) yields
\[
\Bigl|\log\tfrac{dP_t}{dR}-\log\tfrac{d\hat{P}_t}{dR}\Bigr|
\le G_{\hat\pi}+G_{P},
\qquad  t \ge 0,
\]
which completes the proof.
\end{proof}

With the uniform envelopes of Lemma~\ref{lem:Uniform_integrable}, we can differentiate the energy along the Sinkhorn-Frank-Wolfe flow. The bounds on the log-densities of $P_t$ and $\hat\pi_t$ ensure the validity of the chain rule for both the congestion term and the entropy, and allow us to justify all limit operations appearing in the dissipation identity. We now proceed with the proof of Theorem~\ref{thm:EDI}.

\begin{proof}[Proof of Theorem \ref{thm:EDI}.]
By Assumption~\ref{Ass:F_full}, $F$ is linear differentiable on $\Pi(\mu,\nu)$ with derivative $\delta F/\delta m$ that is Lipschitz in $(\Pi(\mu,\nu),W_1)$. We have
\[
\frac{d}{dt}F(P_t) = \int_{\mathcal X\times\mathcal Y} \frac{\delta F}{\delta m}(P_t,x,y)\,\dot P_t(dx,dy).
\]
For $t\ge0$, Lemma~\ref{lem:Uniform_integrable} provides uniform $L^1(R)$ envelopes for the log-densities of $P_t$ and $\hat{P}_t$. This allows us to apply dominated convergence theorem to the map $t\mapsto H(P_t\,\|\,R)$ and obtain the chain rule for all $t\ge0$:
\[
\frac{d}{dt}H(P_t\,\|\,R)
=
\frac{d}{dt}\int \rho_t\log\rho_t\,dR
=
\int_{\mathcal X\times\mathcal Y}
\Bigl(1+\log\rho_t(x,y)\Bigr)\,\dot P_t(dx,dy),
\]
where $\rho_t:=\frac{dP_t}{dR}$,
Since both $P_t$ and $\hat{P}_t$ have unit mass and fixed marginals,
\[
\int_{\mathcal X\times\mathcal Y} \dot P_t(dx,dy)
=\int_{\mathcal X\times\mathcal Y} (\hat{P}_t-P_t)(dx,dy)
=0.
\]
Hence the constant term drops out and for the energy function $V$, we obtain
\[
\frac{d}{dt}V(P_t)
=
\int_{\mathcal X\times\mathcal Y}
\Bigl(\frac{\delta F}{\delta m}(P_t,x,y)
+\log \rho_t(x,y)\Bigr)\,(\hat{P}_t-P_t)(dx,dy).
\]
From the first-order optimality condition \eqref{eq:inner-first-order} for 
$\hat{P}_t$, we have
\begin{equation}
\label{eq:first_order_P_t}
\frac{\delta F}{\delta m}(P_t,x,y)
+\log\tfrac{d\hat{P}_t}{dR}(x,y)
+f_t(x)+g_t(y)=c_{P,t}
\quad\text{for $R$-a.e.\ $(x,y)$},
\end{equation}
for some potentials $f_t,g_t$ and constants $c_{P,t}$.
It follows that
\begin{equation}
\label{eq:expr}
\frac{\delta F}{\delta m}(P_t,x,y)
+\log\tfrac{dP_t}{dR}(x,y)
=
\log\tfrac{dP_t}{d\hat{P}_t}(x,y)
-f_t(x)-g_t(y)+c_{P,t}.
\end{equation}
Using that 
$P_t,\hat{P}_t\in\Pi(\mu,\nu)$ share the same marginals, we obtain
\begin{equation}
\label{eq:marginal-condition}
\begin{aligned}
\int_{\mathcal{X\times Y}} f_t(x)(\hat{P}_t-P_t)(dx,dy) & =\int_{\mathcal{X\times Y}} f_t(x)(\mu-\mu)(dx)=0,\\
\int_{\mathcal{X\times Y}} g_t(y)(\hat{P}_t-P_t)(dx,dy) & =\int_{\mathcal{X\times Y}} g_t(y)(\nu-\nu)(dy)=0,\\
\int_{\mathcal{X\times Y}} c_{P,t}(\hat{P}_t-P_t) &= c_{P,t}(1-1)=0.
\end{aligned}
\end{equation}
Substituting expression \eqref{eq:expr} into the derivative of $V(P_t)$, we get
\begin{equation}
\label{eq:energy_dissipation}
\frac{d}{dt}V(P_t) = \int_{\mathcal X\times\mathcal Y}
\log\tfrac{dP_t}{d\hat{P}_t}(x,y)\,(\hat{P}_t-P_t)(dx,dy) = -\Bigl(H(\hat{P}_t\,\|\,P_t)+H(P_t\,\|\,\hat{P}_t)\Bigr),
\end{equation}
which establishes the energy dissipation identity and the monotonic decay of $V(P_t)$. Since $H(\hat{P}_t\|P_t)\ge 0$, we have
\begin{equation}
\label{eq:dV_inequlity}
    \frac{dV(P_t)}{dt}\le -\,H(P_t\|\hat{P}_t).
\end{equation}
Next we relate $H(P_t\,\|\,\hat{P}_t)$ to the energy gap $V(P_t)-V(\pi)$.
Using the entropy decomposition,
\[
H(P_t\,\|\,\hat{P}_t)
=
H(P_t\,\|\,R)-H(\hat{P}_t\,\|\,R)
-\int_{\mathcal X\times\mathcal Y}
\log\tfrac{d\hat{P}_t}{dR}(x,y)\,(P_t-\hat{P}_t)(dx,dy),
\]
and substituting the inner optimality condition for $\log\tfrac{d\hat{P}_t}{dR}$ in \eqref{eq:first_order_P_t}, the terms involving $f_t,g_t,c_t$ vanish by the
same marginal argument as \eqref{eq:marginal-condition}. We obtain
\[
H(P_t\,\|\,\hat{P}_t)
=
H(P_t\,\|\,R)-H(\hat{P}_t\,\|\,R)
+\int_{\mathcal X\times\mathcal Y}
\frac{\delta F}{\delta m}(P_t,x,y)\,(P_t-\hat{P}_t)(dx,dy).
\]
By minimality of $\hat{P}_t$ in the inner problem \eqref{eq:inner-problem},
\[
H(\hat{P}_t\,\|\,R)
+\int_{\mathcal X\times\mathcal Y}\frac{\delta F}{\delta m}(P_t)\,d\hat{P}_t
\;\le\;
H(\pi\,\|\,R)
+\int_{\mathcal X\times\mathcal Y}\frac{\delta F}{\delta m}(P_t)\,d\pi,
\quad\forall\,\pi\in\Pi(\mu,\nu),
\]
hence
\[
H(P_t\,\|\,\hat{P}_t)
\;\ge\;
H(P_t\,\|\,R)-H(\pi\,\|\,R)
+\int\frac{\delta F}{\delta m}(P_t)\,(P_t-\pi).
\]
By convexity of $F$ in Assumption \ref{Ass:F_full},
\[
F(P_t)-F(\pi)
\;\le\;
\int\frac{\delta F}{\delta m}(P_t)\,(P_t-\pi),
\]
and therefore, for all $\pi\in\Pi(\mu,\nu)$,
\[
H(P_t\,\|\,\hat{P}_t)
\;\ge\;
\bigl[H(P_t\,\|\,R)+F(P_t)\bigr]
-\bigl[H(\pi\,\|\,R)+F(\pi)\bigr]
= V(P_t)-V(\pi).
\]
Taking the infimum over $\pi$ yields
\begin{equation}\label{eq:H-lower-Vgap}
H(P_t\,\|\,\hat{P}_t)
\;\ge\;
V(P_t)-\inf_{\pi\in\Pi(\mu,\nu)}V(\pi)
=
V(P_t)-V_\star.
\end{equation}
Combining \eqref{eq:dV_inequlity} and \eqref{eq:H-lower-Vgap} gives
\[
\frac{d}{dt}\bigl(V(P_t)-V_\star\bigr)
\;\le\;
-\,H(P_t\,\|\,\hat{P}_t)
\;\le\;
-\bigl(V(P_t)-V_\star\bigr),
\qquad t\ge0.
\]
By Grönwall's inequality,
\begin{equation}\label{eq:energy-exp-decay}
V(P_t)-V_\star
\;\le\;
e^{-t}\bigl(V(P_0)-V_\star\bigr),
\qquad t\ge0.
\end{equation}
In particular, $V(P_t)\downarrow V_\star$ exponentially fast.
\medskip
\end{proof}

\subsection{Proof of  Theorem \ref{thm:exp-conv}}
To establish Theorem~\ref{thm:exp-conv}, we first verify two lemmas of the variational problem.

\begin{lemma}[Bounded density of the minimizer $\pi*$]
\label{lem:pistar-bounded}
Let Assumptions~\ref{ass:bounded_domain}, \ref{Ass:F_full}, and \ref{ass:init-feasible} hold. Let $\pi^\star\in\Pi(\mu,\nu)$ be the unique minimizer of $V(\pi)$ given by Proposition \ref{prop:exist-uniq-minimizer}. 
Then there exist constants $0<a_-\le a_+<\infty$ such that
\[
a_- \;\le\; \frac{d\pi^\star}{dR} \;\le\; a_+
\quad R\text{-a.e.\ on }\mathcal X\times\mathcal Y.
\]
\end{lemma}

\begin{proof}[Proof of Lemma \ref{lem:pistar-bounded}]
Let $(P_t)_{t\ge0}$ be the SFW flow defined by~\eqref{eq:FW-ODE} with initial condition $P_0$. By Lemma~\ref{lem:Uniform_integrable} and Assumption~\ref{ass:init-feasible}, there exist constants $0<a_-\le a_+<\infty$ such that, for all $ t \ge 0$,
\begin{equation}
\label{eq:Pt-bounds}
a_- \;\le\; \frac{dP_t}{dR} \;\le\; a_+
\quad R\text{-a.e.\ on }\mathcal X\times\mathcal Y.
\end{equation}
Take a sequence $(t_n)_{n\ge1}$ with $t_n\to\infty$ and $t_n\ge 0$ for all $n$. Since $\Pi(\mu,\nu)$ is weakly compact, there exists a subsequence $(t_{n_k})_{k\ge1}$ and a limit point $\tilde P\in\Pi(\mu,\nu)$ such that
\[
P_{t_{n_k}} \rightharpoonup \tilde P
\quad\text{in }\mathcal P(\mathcal X\times\mathcal Y)\ \text{as }k\to\infty.
\]
On the other hand, Theorem~\ref{thm:EDI} implies that $V(P_t)\downarrow V(\pi^\star)$ as $t\to\infty$, hence
\[
V(\tilde P) \le \liminf_{k\to\infty}V(P_{t_{n_k}}) = V(\pi^\star).
\]
Since $\pi^\star$ is the unique minimizer of $V$ on $\Pi(\mu,\nu)$ by Proposition \ref{prop:exist-uniq-minimizer}, we obtain
\[
V(\tilde P)=V(\pi^\star)
\quad\text{and therefore}\quad
\tilde P=\pi^\star.
\]
For each $k$, the measure $P_{t_{n_k}}$ is absolutely continuous with respect to $R$ and its density satisfies the uniform bounds \eqref{eq:Pt-bounds}. Since the densities $\frac{dP_{t_{n_k}}}{dR}$ are uniformly bounded in $L^\infty(R)$ and $P_{t_{n_k}}\rightharpoonup\pi^\star$ weakly, a standard compactness argument in $L^\infty(R)$ shows that the weak limit $\pi^\star$ is absolutely continuous with respect to $R$, and its Radon–Nikodym derivative inherits the bounds \eqref{eq:Pt-bounds}, that is
\[
a_- \;\le\; \frac{d\pi^\star}{dR} \;\le\; a_+
\quad R\text{-a.e.}
\]
\end{proof}
\begin{lemma}[Fixed-point identity for the minimizer]
\label{lem:fixed-point}
Let Assumptions~\ref{ass:bounded_domain}, \ref{Ass:F_full}, and \ref{ass:init-feasible} hold. Let $\pi^\star$ be the unique minimizer of $V$ given by Proposition~\ref{prop:exist-uniq-minimizer}. Then $\pi^\star$ is a fixed point of the entropic best-response map:
\[
\hat{\,\pi^\star} = \pi^\star.
\]
Equivalently, $\pi^\star$ satisfies the first-order optimality system
\begin{equation}
\label{eq:pi_star_first_order_restate}
\frac{\delta F}{\delta m}(\pi^\star,x,y)
+ \log\!\Bigl(\tfrac{d\pi^\star}{dR}(x,y)\Bigr)
+ f^\star(x) + g^\star(y) = c^\star
\quad\text{for $R$-a.e.\ }(x,y),
\end{equation}
for some measurable potentials $f^\star,g^\star$ and constant $c^\star$, together with $\pi^\star\in\Pi(\mu,\nu)$.
\end{lemma}

\begin{proof}[Proof of Lemma~\ref{lem:fixed-point}]
Consider the SFW flow started at the minimizer: let $(Q_t)_{t\ge0}$ solve
\[
\dot Q_t = T(Q_t)-Q_t, \qquad Q_0=\pi^\star,
\]
where $T(P)=\hat P$. By Lemma~\ref{lem:pistar-bounded}, the initial condition $Q_0=\pi^\star$ is absolutely continuous with respect to $R$ with a uniformly bounded density, hence the flow started at $\pi^\star$ enjoys the same uniform $L^\infty$ density bounds and log-density envelopes as the original flow starting from $P_0$. By Proposition~\ref{prop:wellposedness}, the vector field $P\mapsto T(P)-P$ is $W_1$–Lipschitz, hence the flow $t\mapsto Q_t$ is globally well defined. For any $0\le s<t$, integrating the energy–dissipation identity of 
Theorem~\ref{thm:EDI} along $(Q_u)_{u\in[s,t]}$ yields
\begin{equation}
\label{eq:fixed-EDI-integrated}
V(Q_s)-V(Q_t)
=
\int_s^t \Bigl(
H(Q_u\,\|\,\hat Q_u)
+H(\hat Q_u\,\|\,Q_u)
\Bigr)\,du.
\end{equation}
Since $\pi^\star$ is the unique minimizer of $V$, we have $V(Q_t)\ge V(\pi^\star)$ for all $t\ge0$. Specializing \eqref{eq:fixed-EDI-integrated} to $s=0$ and $Q_0=\pi^\star$ gives
\[
0\ge V(Q_0)-V(Q_t) =\int_0^t \Bigl(H(Q_u\,\|\,\hat Q_u)
+H(\hat Q_u\,\|\,Q_u)
\Bigr)\,du.
\]
Thus, for all $t \ge 0$,
\begin{equation}
\label{eq:fixed-integrand-zero}
\int_0^t \Bigl(H(Q_u\,\|\,\hat Q_u)+H(\hat Q_u\,\|\,Q_u)
\Bigr)\,du = 0.
\end{equation}
Since the integrand in \eqref{eq:fixed-integrand-zero} is nonnegative,  fix any sequence $t_n\downarrow 0$ and set
\[
f_n(v):= H\bigl(Q_{t_n v}\,\|\,\hat Q_{t_n v} \bigr)+H\bigl( \hat Q_{t_n v}\,\|\,Q_{t_n v}\bigr) \ge 0.
\]
Changing variables \(u=t_nv\) in \eqref{eq:fixed-integrand-zero} gives
\[
0=\int_0^1 f_n(v)\,dv.
\]
By Fatou’s lemma,
\[
\int_0^1 \liminf_{n\to\infty} f_n(v)\,dv = 0,
\]
hence $\liminf_{n\to\infty} f_n(v)=0$ for a.e. $v\in[0,1]$. Fix \(v_0\in(0,1]\) such that \(f_n(v_0)\to0\) along a subsequence. By continuity of the flow in $W_1$,
\[
Q_{t_n v_0} \xrightarrow[n\to\infty]{W_1} \pi^\star.
\]
Using the Lipschitz continuity of $T$ (Proposition~\ref{prop:wellposedness}),
\[
\hat Q_{t_n v_0} = T(Q_{t_n v_0})
\xrightarrow[n\to\infty]{W_1} T(\pi^\star)=:\hat \pi^\star.
\]
Thus
\[
(Q_{t_n v_0},\hat Q_{t_n v_0})
\to
(\pi^\star,\hat \pi^\star )
\quad\text{in }W_1\times W_1.
\]
By joint lower semicontinuity of relative entropy,
\[
H(\pi^\star\,\|\,\hat \pi^\star )
+
H(\hat \pi^\star \,\|\,\pi^\star)
\;\le\;
\liminf_{n\to\infty} f_n(v_0) = 0.
\]
Each term is nonnegative, hence both vanish:  
\[
H(\pi^\star\,\|\,\hat \pi^\star )
=
H(\hat \pi^\star \,\|\,\pi^\star)
=
0,
\]
and therefore
\[
\hat \pi^\star=\pi^\star.
\]
Finally, the equivalence with the first-order optimality system \eqref{eq:pi_star_first_order_restate} follows directly from Proposition~\ref{prop:inner-problem}.  
\end{proof}

We now combine the above two lemmas to establish the exponential convergence claimed in Theorem~\ref{thm:exp-conv}.

\begin{proof}[Proof of Theorem \ref{thm:exp-conv}]
We finally show that the exponential decay of the energy controls the convergence of $(P_t)_{t\ge0}$ towards $\pi^\star$ in relative entropy. 
By Lemma \ref{lem:fixed-point}, the minimizer $\pi^\star$ is a fixed point of the entropic
best-response map.  By Proposition~\ref{prop:inner-problem}, this is
equivalent to the first-order optimality system for $V$: there exist
measurable potentials $f^\star,g^\star$ and a constant $c^\star\in\mathbb R$
such that
\begin{equation}
\label{eq:pi_star_first_order}
\frac{\delta F}{\delta m}(\pi^\star,x,y)
+\log\!\Bigl(\tfrac{d\pi^\star}{dR}(x,y)\Bigr)
+f^\star(x)+g^\star(y) = c^\star
\quad\text{for $R$-a.e.\ }(x,y),
\end{equation}
together with the marginal constraints $\pi^\star\in\Pi(\mu,\nu)$.

By Lemma~\ref{lem:fixed-point}, $\pi^\star$ is a fixed point of the inner map and satisfies the first-order condition \eqref{eq:pi_star_first_order}, so we obtain for every $t\ge0$:
\begin{equation}\label{eq:H-Pt-pistar-simple}
\begin{aligned}
H(P_t\,\|\,\pi^\star)
&= H(P_t\,\|\,R)-H(\pi^\star\,\|\,R)
   -\int \log\!\Bigl(\tfrac{d\pi^\star}{dR}\Bigr)\,(P_t-\pi^\star) \\
&= H(P_t\,\|\,R)-H(\pi^\star\,\|\,R)
   +\int \frac{\delta F}{\delta m}(\pi^\star,x,y)\,(P_t-\pi^\star)(dx,dy),
\end{aligned}
\end{equation}
where, in the last equality, the contributions of $c^\star,f^\star,g^\star$  vanish upon integration against $P_t-\pi^\star$ thanks to equal mass and fixed marginals. By convexity of $F$ at $\pi^\star$,
\[
F(P_t)-F(\pi^\star)
\;\ge\;
\int \frac{\delta F}{\delta m}(\pi^\star,x,y)\,(P_t-\pi^\star)(dx,dy),
\]
so from \eqref{eq:H-Pt-pistar-simple} we deduce
\begin{equation}
\label{eq:H-V-gap}
\begin{aligned}
    H(P_t\,\|\,\pi^\star)
    &\le \bigl[H(P_t\,\|\,R)+F(P_t)\bigr] -\bigl[H(\pi^\star\,\|\,R)+F(\pi^\star)\bigr] \\
    &= V(P_t)-V(\pi^\star), \qquad t\ge0.
\end{aligned}
\end{equation}
Together with the exponential decay $V(P_t)-V(\pi^\star)\le e^{-t}\bigl(V(P_0)-V(\pi^\star)\bigr)$ from Theorem \ref{thm:EDI}, this yields
\[
H(P_t\,\|\,\pi^\star) \le e^{-t}\bigl(V(P_0)-V(\pi^\star)\bigr), \qquad t\ge0.
\]
In particular, $H(P_t\,\|\,\pi^\star)\to0$ as $t\to\infty$. By Pinsker's inequality,
\[
\|P_t-\pi^\star\|_{\mathrm{TV}}
\;\le\;
\sqrt{2\,H(P_t\,\|\,\pi^\star)}
\;\xrightarrow[t\to\infty]{} 0,
\]
and therefore $P_t$ converges to $\pi^\star$ in total variation and hence in
the weak topology:
\[
\lim_{t\to\infty} H(P_t\,\|\,\pi^\star)=0,
\qquad
\lim_{t\to\infty}\|P_t-\pi^\star\|_{\mathrm{TV}}=0,
\qquad
P_t \Rightarrow \pi^\star\ \text{as }t\to\infty.
\]
\end{proof}

\subsubsection*{Funding}
A.K.'s research is supported by PEPR PDE-AI project. Z.\,R's research is supported by
the Finance For Energy Market Research Centre,  the France 2030 grant (ANR-21-EXES-0003), and PEPR PDE-AI project.
Y.Z's research is supported by 
 CNRS-Imperial ``Abraham de Moivre" International Research Lab in Mathematics.
 
\bibliography{sn-bibliography}

\end{document}